\documentclass[a4paper,12pt]{article}
\usepackage{amsmath,amsthm,amssymb}
\usepackage{tabularx,bm,wrapfig}
\usepackage[dvipdfmx]{graphicx,hyperref}
\newcommand{\GF}{{\mathbb F}}

\newcommand{\Aut}{{\rm Aut}}

\newtheorem{Thm}{Theorem}[section]
\newtheorem{Lem}[Thm]{Lemma}

\theoremstyle{definition}
\newtheorem{Def}[Thm]{Definition}

\begin{document}
\title{An alternative construction of the $G_2(2)$-hexagon}
\author{Koichi Inoue \thanks{E-mail: k-inoue@teac.paz.ac.jp} \\ \footnotesize{Gunma Paz University, 
Gunma-ken 370-0006, Japan.}}
\date{}
\maketitle

\begin{flushright}
\tiny{Last modified: 2025/01/04}
\end{flushright}

\paragraph{Abstract.}
In this note, we give an alternative and explicit construction of the $G_2(2)$-hexagon from a $U_3(2)$-geometry.


\setcounter{section}{+0}
\section{Introduction}
The paper~\cite{inoue} has given an alternative approach to construct the Chevalley group $G_2(2)$ with a certain 2-transitive permutation representation of degree 28 as well as a concrete tool to study the structure of it. In this note, using the quite elementary description, we shall give an alternative construction of the $G_2(2)$-hexagon (i.e., the split Cayley hexagon of order 2), which yields an alternative approach of $G_2(2)$ with two non-equivalent permutation representations of degree 63.

\section{Preliminaries}

A \textit{partial linear space} of \textit{order} $(s,t)$ is an incidence structure ($\mathcal{P}$,$\mathcal{L}$) where $\mathcal{P}$ is a set of \textit{points} and $\mathcal{L}$ is a set of $(s+1)$-subsets of $\mathcal{P}$ (called \textit{lines}) such that each point is contained in exactly $t+1$ lines and any two distinct points are contained in at most one line (and necessarily any two distinct lines meet in at most one point). If $s=t$, then we say \textit{order} $s$ instead of order $(s,s)$. Let $\bm{S}:=(\mathcal{P},\mathcal{L})$ be a partial linear space. 
The \textit{dual} of $\bm{S}$ is the partial linear space $\bm{S}^{D}:=(\mathcal{P}^{D},\mathcal{L}^{D})$ of order $(t,s)$ where $\mathcal{P}^{D}:=\mathcal{L}$ and $\mathcal{L}^{D}:=\left\{ \{l \in \mathcal{L} \mid p \in l\} \mid p \in \mathcal{P}\right\}$. Two lines of $\bm{S}$ are \textit{concurrent} if they meet in a point. The \textit{concurrency graph} of $\bm{S}$ is the graph whose vertex-set is $\mathcal{L}$, where two vertices are adjacent whenever they are concurrent. 

In order to describe generalized hexagons, we use graph-theoretic terminology as follows. A \textit{graph} is an incidence structure $(V,E)$ where $V$ is a set of \textit{vertices} and $E$ is a set of $2$-subsets of $V$ (called \textit{edges}). Take two vertices $x$ and $y$ (possibly $x=y$). We say that $x$ and $y$ are \textit{adjacent} for an edge $\{x,y\}$. A \textit{path} from $x$ to $y$ is a sequence of distinct vertices $(x=x_0,x_1,\dots,x_m=y)$ such that for $i \in \{1,2,\dots,m\}$, $x_{i-1}$ is adjacent to $x_i$, and its \textit{length} is the number of occurrences of edges. If $x$ and $y$ are connected by a path, the \textit{distance} from $x$ to $y$ is the length of the shortest path from $x$ to $y$, and denoted by $\textrm{d}(x,y)$. A graph is \textit{connected} if there is a path connecting each pair of vertices. The \textit{diameter} of connected graph is the maximum value of all the distances. A \textit{cycle} of \textit{length} $m$ is a closed path $(x_0,x_1,\dots,x_m=x_0)$, and the \textit{girth} of a graph is the minimum length of all the cycles.
  
A partial linear space $\bm{S}:=(\mathcal{P},\mathcal{L})$ of order $(s,t)$ is called a \textit{generalized hexagon} of \textit{order} $(s,t)$ if the incidence graph of $\bm{S}$ (i.e., the graph whose vertex-set is $\mathcal{P} \cup\mathcal{L}$ and edges are pairs $\{p,l\}$ such that $p \in \mathcal{P}, l \in \mathcal{L}$ and $p \in l$) has diameter 6 and girth 12, and we abbreviate it as a $\textrm{GH}(s,t)$. Note that the dual of a $\textrm{GH}(s,t)$ is a $\textrm{GH}(t,s)$.
For two generalized hexagons $\bm{S}$ and $\bm{T},$ we define an $\textit{isomorphism}$ $\sigma$ 
from $\bm{S}$ onto $\bm{T}$ to be a one-to-one mapping from the points of $\bm{S}$ onto the points of 
$\bm{T}$ and the lines of $\bm{S}$ onto the lines of $\bm{T}$ such that $p$ is in $l$ if and only if $p^{\sigma}$ is in $l^{\sigma}$ for each point $p$ and each line $l$ of $\bm{S},$ and then say that $\bm{S}$ and $\bm{T}$ are $\textit{isomorphic}$. An \textit{automorphism} of $\bm{S}$ is an isomorphism from $\bm{S}$ to itself. The set of all automorphisms of $\bm{S}$ forms a group, which is called the \textit{automorphism group} of $\bm{S}$ and denoted by $\Aut\bm{S}$.     
A general introduction for generalized hexagons is van Maldeghem~\cite{maldeghem}. See also Brouwer, Cohen and Neumaier~\cite{bcn}, Brouwer and van Maldeghem~\cite{bm} and Thas~\cite{thas}.

\section{The $G_2(2)$-hexagon}
In this only section, suppose temporarily that $V$ is a 6-dimensional vector space over the field $\GF_2$ with a non-degenerate alternating form, which we call an $S_6(2)$-$geometry$. It follows from Taylor~\cite[Exercise 8.1]{taylor} that the projective space $PG(V)$ contains 63 isotropic points (or vectors), 315 totally isotropic lines (abbreviated as \textit{t.i. lines}), 135 totally isotropic planes (abbreviated as \textit{t.i. planes}). Each t.i. lines $l$ contains 3 vectors, each t.i. plane $M$ contains 7 vectors and $|l \cap M|=0,1$ or $3$. An elementary counting argument shows that each t.i. line is exactly contained 3 t.i. planes. It is well known that the isotropic vectors and some of the t.i. lines form a GH(2,2) with the automorphism group $G_2(2)$. For more details on the construction, see \cite[Section 2.4]{maldeghem}. Although there are various ways to call this hexagon~\cite{maldeghem}, we call it the $G_2(2)$-$hexagon$ (or the \textit{Split Cayley hexagon} of $order$ 2). For the other constructions, see \cite[Section 4.1]{bm}, van Maldeghem~\cite{maldeghem3}, Wilson~\cite[Section 4.3]{wilson}, Payne~\cite{payne}, Cameron and Kantor~\cite{cameron-kantor}, Cossidente and King~\cite{cossidente-king} and Bamberg and Durante~\cite{bamberg-durante}. It is also known that there is a unique GH(2,2) up to isomorphism and duality. This result is due to Cohen and Tits~\cite{cohen-tits}.

\section{The Construction}

Let $V$ be a 3-dimensional vector space over the finite field $\GF_{4}$, and $h$ a non-degenerate  hermitian form. Then we consider the $U_3(2)$-geometry $(V,h)$ instead of the $S_6(2)$-geometry, as $V$ is considered as a 6-dimensional vector space over the field $\GF_2$ with the non-degenerate alternating form $s$, where $s(x,y):=h(x,y)+h(y,x)$ for all $x,y \in V$ (see \cite[Exercise 10.14]{taylor}), and in turn give an alternative and explicit construction of the $G_2(2)$-hexagon, which is not appearing elsewhere in the above references. The definitions and the notations of \cite[Preliminaries]{inoue} will be freely used. 
\vspace{1.0\baselineskip}

\noindent
Since $H_0 \cap \tilde{H_0}=\emptyset$ and $|H_{0}|=|\tilde{H_0}|=6$, the set $V_1:= \{x \in V \mid h(x,x)=1\}$ is divided into  
\begin{align*}
&W := \{x \in V_1 \mid [x] \in H_0\} \\
&\text{and} \\
&\tilde{W}:=\{x \in V_1 \mid [x] \in \tilde{H}_0\},
\end{align*}
each consisting of 18 vectors.
\begin{Def}
Define the design $\bm{S}$ with the point-set $\mathcal{P}:=V \setminus \{0\}$ of which the lines consist of three types as follows: 
\begin{itemize}
\item[(i)] $[a]$;
\item[(ii)] $L(a):=\left\{x \in W \mid h(a,x)=0 \ \text{and} \ a+x \in W\right\} \cup \{a\}$; 
\item[(iii)] $\tilde{L}(a) = \left\{x \in \tilde{W} \mid h(a,x)=0 \ \text{and} \ a+x \in \tilde{W} \right\} \cup \{a\}$;
\end{itemize}
where $a \in V_0$. 
\end{Def}
\vspace{1.0\baselineskip}

\begin{Lem}\label{lem:4.2}
$\bm{S}$ is a partial linear space of order $(2,2)$.
\end{Lem}
\begin{proof}
For $a,b \in V_0$, $a=b$ if and only if $L(a)=L(b)$ or $\tilde{L}(a)=\tilde{L}(b)$, and so the number of lines of $\bm{S}$ is $63 (=9+27+27)$. It follows from $|[a]^{\perp}\cap H_0|=|[a]^{\perp}\cap \tilde{H_0}|=2$ that $L(a)$ is the form $\{a,u,a+u\}$ for some $u \in W$ and $\tilde{L}(a)$ is the form $\{a,v,a+v\}$ for some $v \in \tilde{W}$. Therefore two vectors of $L(a)$ are orthogonal with respect to $s$. By symmetry, the same is true of $\tilde{L}(a)$. These imply that any line of $\bm{S}$ is a t.i. line in $S_6(2)$-geometry. Thus two distinct lines of $\bm{S}$ meet at most one point. Also, each $a \in V_0$ is contained in $[a], L(a)$ and $\tilde{L}(a)$. For each $u \in W$, the set $\{x \in V_0 \mid h(x,u)=0 \ \text{and} \ x+u \in W\}$ has three vectors, say $a_1, a_2$ and $a_3$. Then $u$ is contained in the three lines $L(a_i)$, for all $i \in \{1,2,3\}$. Similarly, for each $v \in \tilde{W}$, put $\{a_1,a_2,a_3\}= \{x \in V_0 \mid h(x,v)=0 \ \text{and} \ x+v \in \tilde{W}\}$, and $v$ is contained in the three lines $\tilde{L}(a_i)$, for all $i \in \{1,2,3\}$. Hence the lemma follows. 
\end{proof}

\noindent
Suppose that $\GF_4=\GF_2[\omega]$, where $\omega^2+\omega+1=0$.
\begin{Lem}\label{lem:4.3}
For each $x \in \mathcal{P}$, the union of the three lines containing $x$ is a t.i. plane.
\end{Lem}
\begin{proof}
(i) For each $a \in V_0$, the three lines containing $a$ are $[a], L(a)$ and $\tilde{L}(a)$. Putting $L(a):=\{a,u,v\} (v:=a+u)$, we have $\tilde{L}(a)=\{a,\omega a+u,\bar{\omega}a+u\}$. Since $a,\omega a$ and $u$ are $\GF_2$-linearly independent and $h(a,u)=0$, it is easy to see that $[a]\cup L(a)\cup \tilde{L}(a)$ is a t.i. plane.

(ii) For each $u \in W$, the set $\{x \in V_0 \mid h(x,u)=0 \ \text{and} \ x+u \in W\}$ has the three vectors $a_1,a_2$ and $a_3$, we see that the three lines containing $u$ are $L(a_i)$ for all $i \in \{1,2,3\}$. Since $a_1+u,a_2+u$ and $a_3+u$ are mutually orthogonal with respect to $h$, $h(a_i,a_j)=0$ for all $i,j \in \{1,2,3\}$, and so $s(a_i,a_j)=0$ for all $i,j$. Therefore the set $\{a_1,a_2,a_3\}$ is a t.i. line and $a_3=a_1+a_2$. An argument similar to the case (i) shows that $L(a_1)\cup L(a_2)\cup L(a_3)$ is a t.i. plane.

(iii) For each $u \in \tilde{W}$, in a similar way to the case (ii), we see that the three lines containing $u$ are $\tilde{L}(b_1), \tilde{L}(b_2)$ and $\tilde{L}(b_3)$, where $\{b_1,b_2,b_3\}:=\{x \in V_0 \mid h(x,u)=0 \ \text{and} \ x+u \in \tilde{W}\}$, and $\tilde{L}(b_1)\cup \tilde{L}(b_2)\cup \tilde{L}(b_3)$ is a t.i. plane.
This proves the lemma.
\end{proof}

\begin{Lem}\label{lem:4.4} \mbox{}
Take $a,b \in V_0$ with $h(a,b)=1$. If $|[a,b]\cap H_0|=2$, then there exists $u \in W$ such that $\{[u+a],[u+b]\} \subset H_0$. 
\end{Lem}
\begin{proof} Putting $\langle u \rangle:=\langle a,b \rangle^{\perp}$, we have from $|[a,b]\cap H_0|=2$ that $[u] \in H_0$. By $|[u,a]\cap H_0|=2$, there is $\alpha \in \GF_{4}^{\times}$ such that $[\alpha u+a] \in H_0$, and $h(\alpha u+a, \alpha u+b)=1+1=0$. Therefore $[\alpha u+b] \in H_0$. Hence the lemma follows.
\end{proof}

\begin{Lem}\label{lem:4.5} \mbox{}
The concurrency graph of $\bm{S}$ is connected. 
\end{Lem}
\begin{proof}
By symmetry of $H_0$ and $\tilde{H}_0$, it is enough to check the connectivity of two lines of the following four types: 
\begin{itemize}
\item[i)] $[a]$ and $[b]$;
\item[ii)] $[a]$ and $L(b)$;
\item[iii)] $L(a)$ and $\tilde{L}(b)$;
\item[iv)] $L(a)$ and $L(b)$; 
\end{itemize}
where $a,b \in V_0$.\\
i) By a scalar multiple of $a$, we may choose $a$ so that $h(a,b)=1$, and put $\langle u \rangle:=\langle a,b \rangle^{\perp}$. 
\begin{itemize}
\item[(i)] If $[u] \in H_0$, then by Lemma~\ref{lem:4.4} we can choose $u$ so that $\left\{ [u+a],[u+b]\right\} \subset H_0$, so $u \in L(a)\cap L(b)$. Therefore there is the path $\left([a],L(a),L(b),[b]\right)$ from $[a]$ to $[b]$. 
\item[(ii)] If $[u] \in \tilde{H}_0$, then we can choose $u$ so that $\{[u+a],[u+b]\} \subset \tilde{H}_0$ as in Lemma~\ref{lem:4.4}, so $u \in \tilde{L}(a)\cap \tilde{L}(b)$. Therefore there is the path $([a],\tilde{L}(a),\tilde{L}(b),[b])$ from $[a]$ to $[b]$.
\end{itemize}
ii) If $b \in [a]$, then $|[a]\cap L(b)|=1$, so we may assume $b \notin [a]$ and choose $a$ so that $h(a,b)=1$. Put $\langle u \rangle:=\langle a,b \rangle^{\perp}$.
\begin{itemize}
\item[(i)] If $[u] \in H_0$, then we can choose $u$ so that $u \in L(a)\cap L(b)$ as in the statement i)(i). Therefore there is the path $([a],L(a),L(b))$ from $[a]$ to $L(b)$. 
\item[(ii)] If $[u] \in \tilde{H}_0$, then we can choose $u$ so that $u \in \tilde{L}(a)\cap \tilde{L}(b)$ as in the statement i)(ii). Also, $\tilde{L}(b)\cap L(b)=\{b\}$, and therefore there is the path $([a],\tilde{L}(a),\tilde{L}(b),L(b))$ from $[a]$ to $L(b)$.
\end{itemize}
iii) If $a=b$, then $L(a)\cap \tilde{L}(b)=\{a\}$. If $a \ne b$ and $b \in [a]$, then there is the path $(L(a),[a],\tilde{L}(b))$ from $L(a)$ to $\tilde{L}(b)$. Therefore we may assume that $b \notin [a]$. Then there exists $\alpha \in \GF_{4}^{\times}$ such that $h(\alpha a,b)=1$. Put $\langle u  \rangle:=\langle a^{\prime},b \rangle^{\perp}$, where $a^{\prime}:=\alpha a$.  
\begin{itemize}
\item[(i)] If $[u] \in H_0$, then we can choose $u$ so that $[u+a^{\prime}] \in H_0$, so $L(a^{\prime})=\{a^{\prime},u,a^{\prime}+u\}$. Since $h(b+u,a^{\prime}+u)=0$, we have $L(b)=\{b,u,b+u\}$. Therefore there is the path $(L(a),[a],L(a^{\prime}),L(b),\tilde{L}(b))$ from $L(a)$ to $\tilde{L}(b)$. 
\item[(ii)] If $[u] \in \tilde{H}_0$, then we can choose $u$ so that $[u+b] \in \tilde{H}_0$, so $\tilde{L}(b)=\{b,u,b+u\}$. Since $h(a^{\prime}+u,b+u)=0$, we have $\tilde{L}(a^{\prime})=\{a^{\prime},u,a^{\prime}+u\}$. Therefore there is the path $(L(a),[a],\tilde{L}(a),\tilde{L}(b))$ from $L(a)$ to $\tilde{L}(b)$. 
\end{itemize}
iv) If $b \in [a]$, then there is the path $(L(a),[a],L(b))$ from $L(a)$ to $L(b)$, so we may assume $b \notin [a]$. Then there exists $\alpha \in \GF_{4}^{\times}$ such that $h(\alpha a,b)=1$. Put $\langle u \rangle:=\langle a^{\prime},b \rangle^{\perp}$, where $a^{\prime}:=\alpha a$.
\begin{itemize}
\item[(i)] If $[u] \in H_0$, then we can choose $u$ so that $u \in L(a^{\prime})\cap L(b)$ as in the statement i)(i). Therefore there is the path $(L(a),[a],L(a^{\prime}),L(b))$ from $L(a)$ to $L(b)$.
\item[(ii)] If $[u] \in \tilde{H}_0$, then we can choose $u$ so that $u \in \tilde{L}(a^{\prime})\cap \tilde{L}(b)$ as in the statement i)(ii). Therefore there is the path $(L(a),[a],\tilde{L}(a^{\prime}),\tilde{L}(b),L(b))$ from $L(a)$ to $L(b)$. 
\end{itemize} 
This completes the proof.
\end{proof}

For the following theorem proved by Cuypers and Steinbach~\cite[Theorem 1.1]{cuypers-steinbach}, we refer to \cite[Theorem 1.3]{bamberg-durante}. Let $Q(6,2)$ be an orthogonal polar space in a projective space $PG(6,2)$.

\begin{Thm}\label{thm:4.6}
Let $\mathcal{L}$ be a set of lines of $Q(6,2)$ such that every point of $Q(6,2)$ is incident with $3$ lines of $\mathcal{L}$ spanning a plane, and such that the concurrency graph of $\mathcal{L}$ is connected. Then the points of $Q(6,2)$ together with $\mathcal{L}$ define a generalized hexagon isomorphic the split Cayley hexagon of order $2$. 
\end{Thm}

Since it is known that $Q(6,2)$ is isomorphic to the symplectic polar space $W(5,2)$, from Lemmas \ref{lem:4.2}, \ref{lem:4.3} and \ref{lem:4.5}, the following theorem holds:

\begin{Thm}\label{thm:4.7}
$\bm{S}$ is the $G_2(2)$-hexagon. 
\end{Thm}




\begin{thebibliography}{99}



\bibitem{bamberg-durante}
J.~Bamberg and N.~Durante, 
Low dimensional models of the finite split Cayley hexagon, 
\emph{Contemp. Math. vol.579: Theory and applications of finite fields},   
Amer. Math. Soc., 2012, 1-19.

\bibitem{bcn}
A.~E.~Brouwer, A.~M.~Cohen and A.~Neumaier, 
\emph{Distance-regular graphs},
Springer-Verlag, Berlin (1989).


\bibitem{bm}
A.~E.~Brouwer and H.~van Maldeghem, 
\emph{Strongly regular graphs},
Encyclopedia of Mathematics and its Applications \textbf{182}, 
Cambridge Univ. Press, 2022.



\bibitem{cameron-kantor}
P.~J.~Cameron and W.~M.~Kantor, 
2-transitive and antiflag transitive collineation groups of finite projective spaces,  
J. Algebra \textbf{60} (1979), 384-422.


\bibitem{cameron-lint} 
P.~J.~Cameron and J.~H.~van~Lint, 
\emph{Designs, Graphs, Codes and their links},
London Mathematical Society Student Text \textbf{22}, 
Cambridge Univ. Press, 1991.

\bibitem{cohen-tits}
A.~M.~Cohen and J.~Tits, On generalized hexagons and a near octagon whose lines have three points, 
Europ. J. Combin. \textbf{6} (1985), 13-27.

\bibitem{atlas} 
J.~H.~Conway, R.~T.~Curtis, S.~P.~Norton, R.~A.~Parker and R.~A.~Wilson, 
\emph{ATLAS of Finite Groups}, 
Clarendon Press, Oxford, 1985.

\bibitem{cossidente-king}
A.~Cossidente and O.~H.~King,
On the geometry of the exceptional group $G_{2}(q)$, $q$ even, 
Des. Codes Cryptogr. \textbf{47} (2008), 145-157.








\bibitem{cuypers-steinbach}
H.~Cuypers and A.~Steinbach,
Near hexagons and triality,  
Beitr$\ddot{a}$ge Algebra Geom. \textbf{45}(2) (2004), 569-580.

\bibitem{inoue}
K.~Inoue, 
An alternative construction of the Hermitian unital 2-(28,4,1) design, J. Combin. Designs \textbf{30} (2022), 752-759. 







\bibitem{payne}
S.~E.~Payne, 
A geometric representation of certain generalized hexagons in $PG(3,s)$, 
J. Combin. Theory \textbf{11} (1971), 181-191.

\bibitem{taylor}
D.~E.~Taylor, 
\emph{The Geometry of the Classical Groups}, 
Sigma Series in Pure Mathematics \textbf{9}, 
Heldermann Verlag, Berlin, 1992.



\bibitem{thas}
J.~A.~Thas, 
Generalized polygons, 
\emph{Handbook of Incidence Geometry: Buildings and Foundations} (ed. F.~Buekenhout), 
North-Holland, 1995, 383-431.


\bibitem{maldeghem}
H.~van Maldeghem, 
\emph{Generalized polygons}, 
Birkh$\ddot{\mathrm{a}}$user Verlag, Basel, 1998.


\bibitem{maldeghem3}
H.~van Maldeghem, 
Some constructions of small generalized polygons,
J. Combin. Theory (A) \textbf{96} (2001), 162-179.


\bibitem{wilson}
R.~A.~Wilson,
\emph{The Finite Simple Groups}, 
Graduate Texts in Mathematics \textbf{251},
Springer, London, 2009.


\end{thebibliography}
\end{document}